\documentclass[a4paper,11pt,reqno, draft]{amsart}

\usepackage[latin1]{inputenc}
\usepackage[english]{babel}
\usepackage{amsfonts}
\usepackage{amssymb,amsmath, times}
\usepackage{graphicx}
\usepackage[pdftex,usenames,dvipsnames]{color}
\usepackage{xcolor}
\usepackage{bbm}

\newtheorem{theo}{Theorem}[section]
\newtheorem{lem}[theo]{Lemma}

\newtheorem{cor}[theo]{Corollary}

\theoremstyle{definition}
\newtheorem{defi}{Definition}

\theoremstyle{remark}
\newtheorem{rem}[theo]{Remark}

\numberwithin{equation}{section}

\def\ee{\varepsilon}
\def\vv{\varphi}

\def\la{\langle}
\def\ra{\rangle}

\def\Qbar{\hbox{\sl Q\kern-.45em{\vrule height.63em width.05em depth-.033em}}~}
\def\hrf{\hbox to.75in{\hrulefill}}

\def\norm#1{\left\Vert#1\right\Vert}

\def\tx{\widetilde{x}}
\def\tvv{\widetilde{\vv}}
\def\hvv{\widehat{\vv}}
\def\hx{\widehat{x}}
\def\tT{\widetilde{T}}
\def\weak{{\omega}}
\def\arg{{\rm arg}}

\title{The Bishop-Phelps-Bollob\'as property for numerical radius in $\ell_1(\mathbb{C})$}
\author{Antonio J. Guirao and Olena Kozhushkina}
\thanks{The research of the first named author was supported in part by MICINN and FEDER (project MTM2008-05396), by Fundaci\'{o}n S\'{e}neca (project 08848/PI/08), by Generalitat Valenciana (GV/2010/036), and by Universidad Polit\'ecnica de Valencia (project PAID-06-09-2829). The research of the second named author is supported by Kent State University}

\address{IUMPA, Universidad Polit\'ectnica de Valencia, 46022, Valencia, Spain}\email{anguisa2@mat.upv.es}
\address{{Dept.} of Mathematical Sciences, Kent State University, Kent OH 44242, USA}\email{okozhush@math.kent.edu}



\begin{document}

\begin{abstract}
We show that the set of bounded linear operators from $X$ to $X$ admits a Bishop-Phelps-Bollob\'as type theorem for numerical radius whenever $X$ is $\ell_1(\mathbb{C})$ or $c_0(\mathbb{C})$. As an essential tool we provide two constructive versions of the classical Bishop-Phelps-Bollob\'as theorem for $\ell_1(\mathbb{C})$.\end{abstract}

\maketitle

\section{Introduction}

The Bishop-Phelps theorem states that {norm attaining} functionals on a Banach space $X$ are dense in its dual space $X^*$. In 1970, B. Bollob\'as extended this result in a quantitative way in order to work on problems related to the numerical range of an operator~\cite{Bollob}. One of the versions of his extension is presented below:

\begin{theo}\label{BPB} Let $X$ be a Banach space. Given $\ee>0$, if $x\in X$, $x^*\in {X^*}$ with $\norm{x}=\norm{x^*}=1$ and $x^*(x)\geq1-\frac{\ee^2}{2}$, then {there exist} elements $x_0\in X$ and $x^*_0\in {X^*}$ such that $\norm{x_0}=\norm{x^*_0}= x^*_0(x_0)=1$,
    \[
       \norm{x-x_0}\leq\ee  \text{ and } \|x^*-x^*_0\|\leq\ee.
    \]
\end{theo}
However, the known proofs of this fact {have} an existence nature --they are based on Hahn-Banach extension theorem, the Ekeland variational principle or Br\o ndsted-Rockafellar {principle}. In this paper we construct, as {a} necessary tool for our main results, explicit expressions of the approximating pair $(x_0,x^*_0)$ when $X=\ell_1(\mathbb{C})$ --see Theorems~\ref{BPB-l1} and \ref{BPB-l1-improv}.

Paralleling the research of {norm attaining} operators initiated by Lindenstrauss in \cite{Lindens}, B. Sims raised the question of the norm denseness of the set of numerical radius attaining operators --see \cite{Sims}. Partial positive results have been proved. We emphasize for their importance the results of M. Acosta in her Ph. D. thesis~\cite{AcostaThesis}, where a systematic study of the problem {was initiated}, {the renorming result in~\cite{ARenorming},} and joint findings of this author with R. Pay\'{a}~\cite{APsecondadj, APRNP}.  Prior to them, I. Berg and B. Sims in~\cite{BergSims} gave a positive answer for uniformly convex spaces and C. S. Cardassi obtained positive answers for $\ell_1$, $c_0$, $C(K)$, $L_1(\mu)$, and uniformly smooth spaces~\cite{Card-USm,Card-c0,Card-CK}.

Using a renorming of $c_0$, R. Pay\'{a} provided an example of a Banach space $X$ such that the set of numerical radius attaining operators on $X$ is not norm dense, answering in the negative Sims' question --see~\cite{Paya}. In the same year, M. Acosta, F. Aguirre, and R. Pay\'{a} in~\cite{AAP} {gave another counterexample: $X=\ell_2 \oplus_{\infty}G$, where $G$ is the Gowers space.}

Recently, M. Acosta \emph{et al.} studied in~\cite{AAGM} a new property, called the \emph{Bishop-Phelps-Bollob\'as property for operators}, BPBp for short. A pair of Banach spaces $(X,Y)$ has the BPBp if a {``Bishop-Phelps-Bollob\'as''} type theorem can be proved for the {set of} operators from $X$ to $Y$. This property implies, in particular, that the norm attaining operators from $X$ to $Y$ are dense in the whole space of continuous linear operators $\mathfrak{L}(X,Y)$. However, { as shown in~\cite{AAGM}, the converse is not true}. Consequently, the BPB property is more than a quantitative tool for studying the density of norm attaining operators.

We investigate here an {analogue of} the Bishop-Phelps-Bollob\'as property for operators but in relation with {numerical} radius attaining operators. { We call it} the \emph{Bishop-Phelps-Bollob\'as property for numerical radius}{, BPBp-$\nu$ for short}. The relation between norm attaining and numerical radius attaining operators is far from being clear, although the existence of an interconnection is evident. Accordingly, our {goals} in this paper {are} to define this new property --see Definition~\ref{BPBp} below-- and to show that $\ell_1(\mathbb{C})$ and $c_0(\mathbb{C})$ satisfy it --see Theorems~\ref{BPBP-l1} and ~\ref{theo:BPBPc0}. This brings an extension {as well as} a quantitative version of C. S. Cardassi's results in~\cite{Card-c0}.

Observe that the counterexamples provided in~\cite{AAP} and~\cite{Paya} imply, in particular, that there exist Banach spaces {failing} the Bishop-Phelps-Bollob\'as property for numerical radius.

Given a Banach space $(X,\norm{\cdot})$, we denote as usual by $S_X$ and $B_X$, respectively, the unit sphere and the unit ball of $X${. By $X^*$ we represent its dual, endowed with its standard norm $\norm{x^*}=\sup_{x\in B_X}\{|x^*(x)|\}$} and by $\Pi(X)$ the set
\[
\Pi(X)=\{(x,x^*)\in S_X\times S_{X^*}\colon \;x^*(x)=1\}.
\]
{Given} $x\in S_X$ and $x^*\in S_{X^*}$, we set
\begin{equation*}\label{eq:pi1}
{\pi_1(x^*):=\{x\in S_X\colon \;x^*(x)=1\}.}
\end{equation*}
By $\mathfrak{L}(X)$ we mean the Banach space of all linear and continuous operators from $X$ into $X$ endowed with its natural norm $\norm{T}=\sup_{x\in B_X}\{\norm{Tx}\}$. For a given $T\in \mathfrak{L}(X)$, its \emph{numerical radius} $\nu(T)$ is defined by
\[
\nu(T)=\sup\{|x^*(Tx)|\colon \; (x,x^*)\in\Pi(X)\}.
\]

It is well known that the numerical radius of a Banach space $X$ is a continuous seminorm on $X$ which is, in fact, an equivalent norm when $X$ is complex. In general, there exists a constant $n(X)$, called the \emph{numerical index} of $X${, such that}
\[
n(X)\norm{T}\leq\nu(T)\leq\norm{T},\text{ for all }T\in\mathfrak{L}(X).
\]

Our interest {in this paper} is {in} spaces of numerical index $1$, {$n(X)=1$}, where the norm and the numerical radius {coincide.}
For background in numerical radius {we refer to the monographs~\cite{BD1,BD2} and in numerical index we refer to the survey~\cite{Kad-Mart-Paya}.}

{We} say that $T\in\mathfrak{L}(X)$ attains its numerical radius if there exists $(x,x^*)\in\Pi(X)$ such that $|x^*(Tx)|=\nu(T)$. The set of numerical radius attaining operators will be denoted by $\rm{NRA}(X)\subset\mathfrak{L}(X)$.

\begin{defi}[BPBp-$\nu$]\label{BPBp}
A Banach space $X$ is said to have the \emph{Bishop-Phelps-Bollob\'as property for numerical radius} if {for every {$0<\varepsilon<1$}, there exists $\delta>0$} such that for {a given} $T\in\mathfrak{L}(X)$ with $\nu(T)=1$ {and a} pair $(x,x^*)\in\Pi(X)$ {satisfying $|x^*(Tx)|\geq1-\delta$}, there exist $S\in \mathfrak{L}(X)$ with $\nu(S)=1$,
and a pair $(y,y^*)\in\Pi(X)$ such that
\begin{equation}\label{eq:defBPBpnu}
{\nu(T-S)\leq\ee,\,\, \norm{x-y}\leq\ee,\,\, \norm{x^*-y^*}\leq\ee\,\,\text{ and }\,\,|y^*(Sy)|=1.}
\end{equation}
\end{defi}

Observe that if $X$ is a Banach space with {$n(X)=1$, then} the seminorm {$\nu(\cdot)$} can be replaced by $\|\cdot\|$ in the definition above. {Note that all} the spaces studied in this paper {have} numerical index $1$.

\subsection*{Notation and terminology.}

Throughout this paper $\arg(\cdot)$ stands for the function which sends a non zero complex number $z$ to the unique  $\arg(z)\in[0,2\pi)$ such that $z=|z|e^{\arg(z)i}$. For convenience we extend the function to $\mathbb{C}$ by writing $\arg(0)=0$. Following the standard notation, let ${\rm Re}(z)$ and ${\rm Im}(z)$ be, respectively, the real and imaginary part of {the} complex number $z\in\mathbb{C}$.

{All along sections \ref{sec:ConstructiveVersions} to \ref{sec:BPBp-nuProperty-c0}, the} spaces $\ell_1$, $\ell_\infty$, and $c_0$ {stand} respectively for $\ell_1(\mathbb{C})$, $\ell_\infty(\mathbb{C})$, and $c_0(\mathbb{C})$. The standard basis of $\ell_1$ is denoted by $\{e_n\}_{n\in\mathbb{N}}$, and its biorthogonal functionals by $\{e_n^*\}_{n\in\mathbb{N}}$. Given a sequence $\xi=(\xi_j)_{j\in\mathbb{N}}{\in\mathbb{C}^{\mathbb{N}}}$ and a complex function $f\colon \mathbb{C}\to\mathbb{C}$ we write $f(\xi)$ meaning the sequence $(f(\xi_j))_{j\in\mathbb{N}}$.

The following sets will be of help in the formulation of the results and proofs.  Given $x=(x_j)_{{j\in\mathbb{N}}}\in\ell_1$, $\vv=(\vv_j)_{{j\in\mathbb{N}}}\in\ell_\infty$ {we define}
\begin{align}
\mathcal{N}_{(x,\,\vv)}&=\{j\in\mathbb{N}\colon \;\vv_j\, x_j=|x_j|\},\label{eq:Nset}\\
{\rm supp}(x)&=\{j\in\mathbb{N}\colon \,|x_j|\neq 0\}.\nonumber
\end{align}
For {$r>0$} we {consider}
\begin{align}
\mathcal{A}_\vv({r})&=\{j\in\mathbb{N}\colon \, |\vv_j|\geq1-{r}\},\label{eq:A}\\
\mathcal{P}_{(x,\vv)}({r})&=\{j\in{\rm supp}(x)\colon \, {\rm Re}(\vv_j\, x_j)\geq(1-{r})|x_j|\}.\label{eq:P}
\end{align}
Observe that $\mathcal{P}_{(x,\vv)}({r})\subset\mathcal{A}_\vv({r})$ and that if $x_j\geq 0$ for all $j\in\mathbb{N}$ --we describe this situation saying that $x$ is \emph{positive}-- then
\begin{equation*}\label{eq:P+positive}
\mathcal{P}_{(x,\vv)}({r})=\{j\in{\rm supp}(x)\colon \, {\rm Re}(\vv_j)\geq(1-{r})\}.
\end{equation*}

For a given set $\Gamma$, a subset $A\subset\Gamma$ and $\mathbb{K}\in\{\mathbb{R},\mathbb{C}\}$, we denote by $\mathbbm{1}_A$ the characteristic function of $A$, that is, the element in $\mathbb{K}^\Gamma$ such that $(\mathbbm{1}_A)_\gamma=1$ if $\gamma\in A$ and $(\mathbbm{1}_A)_\gamma=0$ otherwise.

\section{The Bishop-Phelps-Bollob\'as theorem in $\ell_1(\mathbb{C})$}\label{sec:ConstructiveVersions}

In this section we present two constructive versions of Theorem~\ref{BPB}, which are the main tool in the proof of our Theorems~\ref{BPBP-l1} and~\ref{BPBP-l1-strong}.
\begin{lem}\label{characteriz}
Let $(x,\vv)\in S_{\ell_1}\times S_{\ell_\infty}$. Then $x\in\pi_1(\vv)$ if and only if $\mathcal{N}_{(x,\vv)}=\mathbb{N}$.
\end{lem}

\begin{proof}
Given a pair $(x,\vv)\in S_{\ell_1}\times S_{\ell_\infty}$ satisfying $\mathcal{N}_{(x,\vv)}=\mathbb{N}$, one can compute $\vv(x)=\sum_{j\in\mathbb{N}} \vv_j\, x_j\stackrel{\eqref{eq:Nset}}{=}\sum_{j\in\mathbb{N}} |x_j|=\norm{x}=1$,
which implies that $(x,\vv)\in\Pi(\ell_1)$.

Conversely, let us assume that $(x,\vv)\in\Pi(\ell_1)$ then,
\[
1={\rm Re}(\vv(x))=\sum_{j\in\mathbb{N}} {\rm Re}(\vv_j\, x_j)\leq\sum_{j\in\mathbb{N}} |\vv_j\, x_j|\leq\sum_{j\in\mathbb{N}} |x_j|=1,
\]
which implies that ${\rm Re}(\vv_j\, x_j)=|\vv_j\, x_j|=|x_j|\mbox{ for } j\in\mathbb{N}$. Therefore, $\vv_j\, x_j=|x_j|$ for {every} $j\in\mathbb{N}$, which finishes the proof.
\end{proof}

Lemma~\ref{characteriz} provides the essential insight into the properties of $\Pi(\ell_1)$ that we need for the proof of Theorems~\ref{BPB-l1} and~\ref{BPB-l1-improv}. {A glance at} Lemma~\ref{characteriz}  gives the following easy result {regarding the norm attaining functionals on $\ell_1$, ${\rm NA}(\ell_1)$.}

\begin{cor}
$\rm{NA}(\ell_1)=\{\vv\in\ell_\infty\colon \, \exists n\in\mathbb{N}\text{ with }|\vv_n|=\|\vv\|\}$.
\end{cor}

The following lemma is an adaptation of~\cite[Lemma~3.3]{AAGM} to our notation.

\begin{lem}\label{key-lemma}
Let $(x,\vv)\in B_{\ell_1}\times B_{\ell_\infty}$ and $0<\delta<1$ such that $\vv(x)\geq1-\delta$. Then, for every $\delta<r<1$ we have $\big\|{\rm Re}\big({e^{\arg(\vv)i}\,x}\big)\cdot\mathbbm{1}_{\mathcal{P}_{(x,\vv)}(r)}\big\|\geq 1-(\delta/r)$.
\end{lem}
\begin{proof}
By assumption, we have that
\begin{align*}
1-\delta &\leq{\rm Re}(\vv(x))=\sum_{j\in\mathbb{N}}{\rm Re}(\vv_j\, x_j)=\sum_{j\in\mathbb{N}}|\vv_j|\,{\rm Re}\big({e^{\arg(\vv_j)i}\, x_j}\big)\\
         &\leq \sum_{\mathcal{P}_{(x,\vv)}(r)}{\rm Re}\big({e^{\arg(\vv_j)i}\, x_j}\big)+(1-r)\sum_{\mathbb{N}\setminus\mathcal{P}_{(x,\vv)}(r)}|x_j|\\
         &\leq r\sum_{\mathcal{P}_{(x,\vv)}(r)}\left|{\rm Re}\big({e^{\arg(\vv_j)i}\, x_j}\big)\right|+(1-r),
\end{align*}
which implies that
\[
{\Big\|}{\rm Re}\big({e^{\arg(\vv)i}\,x}\big)\,\mathbbm{1}_{\mathcal{P}_{(x,\vv)}(r)}{\Big\|}=\sum_{j\in\mathcal{P}_{(x,\vv)}(r)}\left|{\rm Re}\Big({e^{\arg(\vv_j)i}\,x_j}\Big)\right|\geq 1-(\delta/r),
\]
as we wanted to show.
\end{proof}

Observe that the previous lemma implies, in particular, that
\begin{equation*}\label{eq:keylemma}
{\big\|}x\cdot\mathbbm{1}_{\mathcal{P}_{(x,\vv)}(r)}{\big\|}\geq 1-(\delta/r).
\end{equation*}

We present next the two constructive versions of the Bishop-Phelps-{Bollob\'as} theorem.

\subsection{First constructive version}

\begin{theo}\label{BPB-l1}
Given $(x,\,\vv)\in B_{\ell_1}\times B_{\ell_\infty}$ and $0<\ee<1$ such that $\vv(x)\geq1-\frac{\ee^3}{4}$. Then, there exists
$(x_0,\vv_0)\in\Pi(\ell_1)$ such that $\norm{x-x_0}\leq\ee$, $\norm{\vv-\vv_0}\leq\ee$. Moreover, we can {take}
\begin{equation}\label{eq:vector-v1}
x_0:=\big\|x\cdot\mathbbm{1}_{\mathcal{P}_{(x,\vv)}(\ee^2/2)}\big\|^{-1}\cdot x\cdot\mathbbm{1}_{\mathcal{P}_{(x,\vv)}(\ee^2/2)}.
\end{equation}
\end{theo}

\begin{proof}
{Set $P:=\mathcal{P}_{(x,\vv)}(\ee^2/2)$ --see definition~\eqref{eq:P}}. Applying Lemma~\ref{key-lemma} {with $\delta=\ee^2/2$ and $r=\ee$} gives that
\begin{equation}\label{eq:M}
M{:=}\norm{x\cdot\mathbbm{1}_{P}}\geq 1-(\ee/2).
\end{equation}
Let us define
\begin{equation}\label{eq:definicionfi}
{\vv_0:=\vv\cdot\mathbbm{1}_{\mathbb{N}\setminus P}+e^{-\arg(x)i}\cdot\mathbbm{1}_{P}\in S_{\ell_\infty}}
\end{equation}
and
\begin{equation}\label{eq:definicionpunto}
{x_0:=M^{-1}x\cdot\mathbbm{1}_{P}\in S_{\ell_1}.}
\end{equation}

{On one hand}, we can compute
\begin{align*}
\norm{x-x_0}&\stackrel{{\eqref{eq:definicionpunto}}}{=}\norm{x-M^{-1}x\cdot\mathbbm{1}_{P}}=(M^{-1}-1)\norm{x\cdot\mathbbm{1}_{P}}+\norm{x\cdot\mathbbm{1}_{{\mathbb{N}\setminus P}}}\\
            &\stackrel{{\eqref{eq:M}}}{=}(1-M)+\norm{x\cdot\mathbbm{1}_{{\mathbb{N}\setminus P}}}{\stackrel{\|x\|\leq 1}{\leq}} 2-2M{\stackrel{\eqref{eq:M}}{\leq}}\ee,
\end{align*}
and, since the support of $x_0$ is included in $P$ {--this is a consequence of~\eqref{eq:definicionpunto}}, we deduce that
\[
\vv_0(x_0)=\sum_{j\in P}{(\vv_0)_{j}\,(x_0)_{j}}\stackrel{{\eqref{eq:definicionfi}}}{=}\sum_{j\in P}e^{-\arg(x_j)i}\,{(x_0)_{j}}\stackrel{{\eqref{eq:definicionpunto}}}{=}\sum_{j\in P}|{(x_0)_{j}}|{=\norm{x_0}}=1,
\]
which is equivalently expressed as $(x_0,\vv_0)\in\Pi(\ell_1)$.

On the other hand, {using that
\begin{equation}\label{eq:desigCompleja}
|z-1|\leq\sqrt{2(1-{\rm Re}(z))}\text{ for every }z\in\mathbb{C}\text{ such that }|z|\leq1,
\end{equation}
we deduce}
\begin{align*}
\norm{\vv-\vv_0}&\stackrel{{\eqref{eq:definicionfi}}}{=}\sup_{j\in P}\{|\vv_j-{(\vv_0)_{j}}|\}\stackrel{{\eqref{eq:definicionfi}}}{=}\sup_{j\in P}\big\{\big|\vv_j-e^{-\arg(x_j)i}\big|\big\}\\
                &=\sup_{j\in P}\big\{\big|e^{\arg(x_j)i}\,\vv_j-1\big|\big\}
                \stackrel{{\eqref{eq:desigCompleja}}}{\leq}\sup_{j\in P}\left\{\sqrt{2-2\,{\rm Re}\big(e^{\arg(x_j)i}\,\vv_j\big)}\right\}\\
                &\leq\sqrt{2-2(1-\ee^2/2)}=\ee,
\end{align*}
which finishes the proof.
\end{proof}

An immediate consequence of Theorem~\ref{BPB-l1} is the following version of the Bishop-Phelps-Bollob\'as theorem for $\ell_1(\mathbb{C})$.

\begin{cor}\label{coro-BPB-l1}
Let {$0<\ee<1$} and $(x,\,\vv)\in B_{\ell_1}\times B_{\ell_\infty}$ such that $|\vv(x)|\geq1-\frac{\ee^3}{4}$. Then, there exists
$(x_0,\vv_0)\in S_{\ell_1}\times S_{\ell_\infty}$ such that $\norm{x-x_0}\leq\ee$, $\norm{\vv-\vv_0}\leq\ee$ and $|\vv_0(x_0)|=1$.
\end{cor}

\begin{proof}
Apply Theorem~\ref{BPB-l1} to the pair ${\left(e^{-\arg(\vv(x))i}\,x,\vv\right)}$ obtaining $(z_0,\vv_0)$ belonging to $\Pi(\ell_1)$ such that $\norm{e^{-\arg(\vv(x))i}\,x-z_0}\leq\ee$ and  $\norm{\vv-\vv_0}\leq\ee$. Therefore, if we set $x_0{:=}e^{\arg(\vv(x))i}\,z_0$, the pair $(x_0,\vv_0)$ satisfies the conclusions of the corollary.
\end{proof}

\subsection{Second constructive version}

Given a pair $(x,\vv)$ and {$0<\ee<1$}, Theorem~\ref{BPB-l1} ensures the existence of a pair $(x_0,\vv_0)$ --defined by~\eqref{eq:definicionpunto} and \eqref{eq:definicionfi}-- satisfying the conclusions of the Bishop-Phelps-Bollob\'as theorem. However, $\vv_0$ depends on $x$, in fact, on $\arg(x)$. In order to prove Theorem~\ref{BPBP-l1} we {will} need a functional $\vv_0$ depending {only on the given $\ee$ and $\vv$}.  So, we present the following result.

\begin{theo}\label{BPB-l1-improv}
Let $(x,\,\vv)\in B_{\ell_1}\times B_{\ell_\infty}$ and $\,0<\ee<1$ be such that $\vv(x)\geq1-\frac{\ee^3}{60}$. Then there exists
$(x_0,\vv_0)\in\Pi(\ell_1)$ such that $\norm{x-x_0}\leq\ee$, $\norm{\vv-\vv_0}\leq\ee$. Moreover, the functional $\vv_0$ can be defined as
\begin{equation}\label{eq:Funcional}
\vv_0=\vv\cdot\mathbbm{1}_{\mathbb{N}\setminus{\mathcal{A}_{\vv}(\ee^2/20)}}+e^{\arg(\vv)i}\cdot\mathbbm{1}_{\mathcal{A}_{\vv}(\ee^2/20)}.
\end{equation}
\end{theo}
\begin{proof}
Let us consider the isometry $S\colon\ell_1\to\ell_1$ defined  by
	\begin{equation}\label{eq:IsometryS}
	\la e_j^*,Sy\ra=e^{\arg(\vv_j)i}\,y_j,\,\text{ for $y\in\ell_1$ and }\,j\in\mathbb{N}.
	\end{equation}
Set $\tx=Sx$ and $\tvv=\vv\circ S^{-1}$. Then, it is clear that the pair $(\tx,\tvv)$ is in $B_{\ell_1}\times B_{\ell_\infty}$, that $\tvv(\tx)\geq1-\frac{\ee^3}{60}$ and that $\tvv=(|\vv_j|)_{j\in\mathbb{N}}$ is positive. Denote by $A$ and $P$ respectively the sets $\mathcal{A}_{\tvv}(r)$ and $\mathcal{P}_{(\tx,\tvv)}(r)$ {--see definitions~\eqref{eq:A} and~\eqref{eq:P}}, where $r{:=}\frac{\ee^2}{20}$. Let us define
\begin{equation}\label{eq:funcionalAux}
\hvv{:=}\tvv\cdot\mathbbm{1}_{\mathbb{N}\setminus A}+\mathbbm{1}_{A}\in S_{\ell_\infty}
\end{equation}
and
\begin{equation}\label{eq:vectorAux}
\hx{:=}M^{-1}{\rm Re}(\tx)\cdot\mathbbm{1}_{P}\in S_{\ell_1},
\end{equation}
 where $M{:=}\norm{{\rm Re}(\tx)\cdot\mathbbm{1}_{P}}$. {Applying Lemma~\ref{key-lemma} with $\delta=\ee^3/60$ and $r$,} gives that $M\geq 1-\frac{\ee}{3}$. In particular, this means that $P$, and {thus} $A$, are non-empty.

We can compute that
\begin{align}
\norm{\tvv-\hvv}&\stackrel{{\eqref{eq:funcionalAux}}}{=}\sup_{j\in A}\{|\tvv_j-\hvv_{j}|\}\stackrel{{\eqref{eq:funcionalAux}}}{=}\sup_{j\in A}\{|\tvv_j-1|\}\nonumber\\
                &=\sup_{j\in A}\{(1-\tvv_j)\}\stackrel{{\eqref{eq:A}}}{\leq}  r\leq\ee,\label{eq:desigFuncional}
\end{align}
and, since {by~\eqref{eq:P} and \eqref{eq:vectorAux}} the support of $\hx$ is $P\subset A$ --which, in particular, implies that $\hx_{j}>0$ for $j\in P$, we deduce that
\begin{equation}\label{eq:inPI}
\hvv(\hx)=\sum_{j\in P}\hvv_{j}\hx_{j}\stackrel{{\eqref{eq:funcionalAux}}}{=}\sum_{j\in P}\hx_{j}=\sum_{j\in P}|\hx_{j}|=1,
\end{equation}
which is equivalently written as $(\hx,\hvv)\in\Pi(\ell_1)$.

In order to show that $\norm{\tx-\hx}\leq\ee$, let us observe first that
\begin{equation}\label{eq:aux1}
\norm{\tx\cdot\mathbbm{1}_{P}}=\sum_{j\in P}|\tx_j|\geq{\sum_{j\in P}|{\rm Re}(\tx_j)|=} M\geq 1-\frac{\ee}{3},
\end{equation}
from which
\begin{align}
\norm{\tx-\hx}&\stackrel{{\eqref{eq:vectorAux}}}{=}
\norm{\tx-M^{-1}{\rm Re}(\tx)\cdot\mathbbm{1}_{P}}=\norm{\tx\cdot\mathbbm{1}_{\mathbb{N}\setminus P}}+\norm{(\tx-M^{-1}{\rm Re}(\tx))\cdot\mathbbm{1}_{P}}\nonumber\\
            &\stackrel{{\eqref{eq:aux1}}}{\leq}\frac{\ee}{3}+\norm{(\tx-M^{-1}{\rm Re}(\tx))\cdot\mathbbm{1}_{P}}.\label{eq:aux2}
\end{align}
We need a bit more care to estimate the last term in~\eqref{eq:aux2}. From the very definition of $P$, we know that for every $j\in P$ it holds
\begin{equation}\label{eq:aux3}
|\tx_j|\leq(1-r)^{-1}\tvv_j\,{\rm Re}(\tx_j).
\end{equation}
Therefore,
\begin{align}
\norm{(\tx-{\rm Re}(\tx))\cdot\mathbbm{1}_P}&=\sum_{j\in P}|\tx_j-{\rm Re}(\tx_j)|=\sum_{j\in P}|{\rm Im}(\tx_j)|\nonumber\\
                                &=\sum_{j\in P}\sqrt{|\tx_j|^2-{\rm Re}(\tx_j)^2}\nonumber\\
                                &\stackrel{{\eqref{eq:aux3}}}{\leq}\sum_{j\in P}|{\rm Re}(\tx_j)|\sqrt{(1-r)^{-2}-1}\nonumber\\
                                &\leq{\norm{\tx}}\sqrt{(1-r)^{-2}-1}\stackrel{{r=\frac{\ee^2}{20}}}{\leq}\frac{\ee}{3},\label{eq:aux4}
\end{align}
which implies that
\begin{align}
\norm{(\tx-M^{-1}{\rm Re}(\tx))\cdot\mathbbm{1}_{P}}&\leq\norm{(\tx-{\rm Re}(\tx))\cdot\mathbbm{1}_P}+\norm{(1-M^{-1}){\rm Re}(\tx)\cdot\mathbbm{1}_{P}}\nonumber\\
&\stackrel{{\eqref{eq:aux4}}}{\leq}\frac{\ee}{3}+(M^{-1}-1)\norm{{\rm Re}(\tx)\cdot\mathbbm{1}_{P}}\nonumber\\
&=\frac{\ee}{3}+(1-M)\leq\frac{2\ee}{3}.\label{eq:aux5}
\end{align}

Putting together~\eqref{eq:aux2} and~\eqref{eq:aux5}, one obtains
\begin{equation}\label{eq:desigVector}
\norm{\tx-\hx}\leq \frac{\ee}{3}+\norm{(\tx-M^{-1}{\rm Re}(\tx))\cdot\mathbbm{1}_{P}}\leq\ee,
\end{equation}
which finishes the core of the proof.

Now, we define
\begin{equation}\label{eq:transfer}
x_0{:=}S^{-1}\hx\quad\text{and}\quad\vv_0=S^*(\hvv)=\hvv\circ S,
\end{equation}
which by~\eqref{eq:inPI} gives that $\vv_0(x_0)=\hvv(\hx)=1$. Since $S$ and $S^*$ are isometries, we deduce from~\eqref{eq:desigFuncional},~\eqref{eq:desigVector},~\eqref{eq:transfer} and the definition of $\tx$ and $\tvv$  that
\[
\norm{x-x_0}\leq\ee,\text{ }\norm{\vv-\vv_0}\leq\ee.
\]
Therefore, $(x_0,\vv_0)$ is the pair in $\Pi(\ell_1)$ we were looking for.

Bearing in mind~\eqref{eq:transfer}, one computes
\[
(\vv_{0})_j=\vv_0(e_j)\stackrel{{\eqref{eq:transfer}}}{=}\hvv(Se_j)\stackrel{{\eqref{eq:IsometryS}}}{=}\hvv\big(e^{\arg(\vv_j)i}\,e_j\big)=e^{\arg(\vv_j)i}\,\hvv_{j},
\]
which together with~\eqref{eq:funcionalAux} implies that $\vv_0=\vv\cdot\mathbbm{1}_{{\mathbb{N}\setminus A}}+e^{\arg(\vv)i}\cdot\mathbbm{1}_{A}$. {Finally, noting} that $A=\mathcal{A}_{\tvv}(r)=\mathcal{A}_{\vv}(r)$, the validity of~\eqref{eq:Funcional} has been shown.
\end{proof}

{\begin{rem}\label{Rem:BPB}
Observe that the function $\vv_0$ provided by Theorem~\ref{BPB-l1-improv} and defined by~\eqref{eq:Funcional}  only depends on $\ee$ and $\vv$ itself as well as satisfies $\pi_1(\vv)\subset\pi_1(\vv_0)$.
\end{rem}}

\section{BPB property for Numerical Radius in $\ell_1(\mathbb{C})$}\label{sec:BPBp-nuProperty-l1}

As a consequence of Theorems~\ref{BPB-l1} and~\ref{BPB-l1-improv} we show that $\ell_1$ has the Bishop-Phelps-Bollob\'as property for numerical radius.

\begin{theo}\label{BPBP-l1}
Let $T\in S_{\mathfrak{L}(\ell_1)}$, $0<\ee<1$ and $(x,\vv)\in\Pi(\ell_1)$  such that $\vv(Tx)\geq1-{(\ee/9)^{9/2}}$. Then there exist $T_0\in S_{\mathfrak{L}(\ell_1)}$ and $(x_0,\vv_0)\in\Pi(\ell_1)$ such {that}
\begin{equation}\label{eq:TheoThesis}
\norm{T-T_0}\leq\ee,\,\, \norm{x-x_0}\leq\ee,\,\,
\norm{\vv-\vv_0}\leq\ee\,\,\text{and}\,\,\vv_0(T_0x_0)=1.
\end{equation}
\end{theo}
\begin{proof}
First of all, fix $\mu:=\sqrt{\ee^3/240}$. Using a suitable isometry, we can assume that $x$ is positive. In particular, by Lemma~\ref{characteriz} and the definition of $\mathcal{N}_{x,\vv}$ in~\eqref{eq:Nset}, we can assume that $\vv_j=1$ for $j\in{\rm supp}(x)$. Since $\mu^3/4\geq(\ee/9)^{9/2}$, Theorem~\ref{BPB-l1} can be applied to the pair $(x,\,T^*\vv)\in B_{\ell_1}\times B_{\ell_\infty}$ and $\mu$ instead of $\ee$ giving $x_0\in\pi_1(\vv)$ such that $\norm{x-x_0}\leq\mu\leq\ee$. Moreover,  by~\eqref{eq:vector-v1} we know that
\begin{equation}\label{eq:TeoVector}
x_0=\norm{x\cdot\mathbbm{1}_{P}}^{-1}\cdot x\cdot\mathbbm{1}_{P},
\end{equation}
where the non-empty set $P$ is defined by
\begin{equation}\label{eq:TheoP}
{P:=\mathcal{P}_{(x,\,T^*\vv)}(\mu^2/2)=\{j\in{\rm supp}(x)\colon {\rm Re}(T^*\vv(e_j))\geq 1- \mu^2/2\}.}
\end{equation}
In particular, $x_0$ is positive.

{Since $\mu^2/2=\frac{(\ee/2)^3}{60}$, for} each $j\in P$ we {can} apply Theorem~\ref{BPB-l1-improv} to {the pair $(e^{-\arg(\vv(Te_j))i}\, Te_j,\vv)$ and $\ee/2$}  to find  $(z_j,\vv_0)\in\Pi(\ell_1)$ such that
\[
\norm{Te_j-a_jz_j}\leq\ee/2,\quad \norm{\vv-\vv_0}\leq\ee/2
\]
and $\Pi_1(\vv)\subset\Pi_1(\vv_0)$ {--see Remark~\ref{Rem:BPB}}, where $a_j=e^{\arg(\vv(Te_j))i}$.
 {Observe that $\vv_0$ can be chosen independently on $j\in P$ and by~\eqref{eq:Funcional} explicitly written as}
\begin{equation}\label{eq:Funcional-bis}
\vv_0=\vv\cdot\mathbbm{1}_{\mathbb{N}\setminus{\mathcal{A}_{\vv}(\ee^2/80)}}+e^{\arg(\vv)i}\cdot\mathbbm{1}_{\mathcal{A}_{\vv}(\ee^2/80)}.
\end{equation}

Let us define $T_0$ as the unique operator in $\mathfrak{L}(\ell_1)$ such that $T_0e_i=Te_i$ for $i\notin P$ and $T_0e_j=z_j$ for $j\in P$. {Equivalently,}
\begin{equation}\label{eq:operator}
{T_0x=\mathbbm{1}_{\mathbb{N}\setminus P}\cdot Tx+\sum_{j\in P}e_j^*(x)z_j, \text{ for } x\in\ell_1.}
\end{equation}
It is clear {from~\eqref{eq:operator}} that
\[
\norm{T_0}=\sup_{n\in\mathbb{N}}\{\norm{T_0e_n}\}=\max\bigg\{\sup_{j\notin P}\{\norm{Te_j}\},\,\sup_{j\in P}\{\norm{z_j}\}\bigg\}= 1.
\]
{Given $j\in P$, the identity~\eqref{eq:TheoP} ensures that ${\rm Re}(\vv(Te_j))\geq1-\mu^2/2$. Using again the general fact~\eqref{eq:desigCompleja}, we deduce that $|a_j-1|\leq\mu\leq\ee/2$.}

{Therefore,}
\begin{align*}
\norm{T-T_0}&=\sup_{n\in\mathbb{N}}\{\norm{Te_n-T_0e_n}\}=\sup_{j\in P}\{\norm{Te_{j}-z_j}\}\\
          &\leq\sup_{j\in P}\{\norm{Te_{j}-a_jz_j}\}+\sup_{j\in P}\{\norm{a_jz_j-z_j}\}\\
          &\leq\frac{\ee}{2}+\sup_{j\in P}\{|a_j-1|\}\leq\ee.
\end{align*}

Since $x_0\in\pi_1(\vv)$ and $\pi_1(\vv)\subset\pi_1(\vv_0)$, we deduce that $(x_0,\vv_0)$ belongs to $\Pi(\ell_1)$. It remains to
show that $\vv_0(T_0x_0)=1$ {to prove the validity of~\eqref{eq:TheoThesis}. But, since $x_0$ is positive}, we obtain that
\begin{align*}
\vv_0(T_0x_0)&\stackrel{{\eqref{eq:operator}}}{=}\sum_{j\in P} (x_0)_{j}\vv_0(z_j)+\sum_{{j\notin P}}(x_{0})_{j}\vv_0(Te_j)\\
           &\stackrel{{\eqref{eq:TeoVector}}}{=}\sum_{j\in P}(x_{0})_{j}=\sum_{j\in P} |(x_{0})_{j}|=\norm{x_0}=1, \end{align*}
and the proof is over.
\end{proof}

\begin{rem}
We cannot replace the condition $(x,\vv)\in\Pi(\ell_1)$ in Theorem~\ref{BPBP-l1} by the more general $(x,\vv)\in B_{\ell_1}\times B_{\ell_\infty}$. Indeed, let us consider the operator $T\colon \ell_1\to\ell_1$ defined by $Te_j=e_j$ for $j\geq 2$ and $Te_1=e_2$. Take $(e_1,e_2^*)\in B_{\ell_1}\times B_{\ell_\infty}$, $T_0\in\mathfrak{L}(\ell_1)$, and $(x,\vv)\in B_{\ell_1}\times B_{\ell_\infty}$ such that  $\|T-T_0\|\leq\ee$, $\|e_1-x\|\leq\ee$, and $\|e_2^*-\vv\|\leq\ee$. Then
\[
|\vv(x)|\leq|\vv(x)-e_2^*(x)|+|e_2^*(x)-e_2^*(e_1)|+|e_2^*(e_1)|\leq2\ee,
\]
which implies that $(x,\vv)$ cannot be in $\Pi(\ell_1)$.
\end{rem}

\begin{cor}\label{NR-l1}
The Banach space $\ell_1$ has the Bishop-Phelps-Bollob\'as property for numerical radius.
\end{cor}
\begin{proof}
Let us consider $T\in\mathfrak{L}(\ell_1)$ {with $\nu(T)=1$ and $0<\ee<1$.} Let us take a pair $(x,\vv)\in\Pi(\ell_1)$ such that $|\vv(Tx)|\geq1-({\ee/9})^{\frac{9}{2}}$. In fact, we can assume that $\vv(Tx)\geq1-({\ee/9})^{\frac{9}{2}}$; otherwise, we proceed with $\tT=e^{-\arg(\vv(Tx))i}\,T$. Then Theorem~\ref{BPBP-l1} gives the existence of an operator $T_0{\in S_{\mathfrak{L}(\ell_1)}}$ and a pair $(x_0,\vv_0){\in\Pi(\ell_1)}$ that satisfy {conditions in~\eqref{eq:TheoThesis}, which are precisely the requirements~\eqref{eq:defBPBpnu} in} Definition~\ref{BPBp}.
\end{proof}
{\begin{cor}[{\cite{Card-c0}}]
The set ${\rm NRA}(\ell_1)$ is dense in $\mathfrak{L}(\ell_1)$.
\end{cor}
}
\section{BPB property for Numerical Radius in $c_0(\mathbb{C})$}\label{sec:BPBp-nuProperty-c0}

Theorem~\ref{BPBP-l1} allows us to show that $c_0$ has the Bishop-Phelps-Bollob\'as property for numerical radius as well. Indeed, we rely on the fact that our constructions in $\ell_1$ can be dualized.

\begin{theo} \label{theo:BPBPc0}
Let $T\in S_{\mathfrak{L}(c_0)}$, $0<\ee<1$ and $(x,\vv)\in \Pi(c_0)$ such that $|\vv(Tx)|\geq1-(\ee/9)^{9/2}$.
Then there exist $S\in S_{\mathfrak{L}(c_0)}$ and $(x_{0},\vv_0)\in \Pi(c_0)$, such {that}
\begin{equation*}
\norm{T-S}\leq \ee,\,\,\norm{x-x_0}\leq \ee,\,\, \norm{\vv-\vv_0}\leq \ee\,\,\text{and}\,\,\vv_0(Sx_{0})=1{.}
\end{equation*}
\end{theo}

\begin{proof} {Throughout this proof we identify the elements in $c_0$ with their image in $\ell_\infty$ through the natural embedding $c_0\to\ell_\infty$.}
The adjoint operator of $T$,  $T^*\colon \ell_1\rightarrow \ell_1$ satisfies
\[
	|x(T^*\vv)|=|T^*(\vv)(x)|=|\vv(Tx)|{\geq}1-(\ee/9)^{9/2}.
\]

Without loss of generality, we can assume that $x(T^*\vv)\geq1-(\ee/9)^{9/2}$. Otherwise, employing techniques from {the proof of} Corollary~\ref{NR-l1}, {define the operator} $\tT=e^{-\arg(x(T^*\vv))i}\,T^*$
and proceed with the proof for $x(\tT\vv)=|x(T^*\vv)|$.

By Theorem~\ref{BPBP-l1}, there exists $T_0\in \mathfrak{L}(\ell_1)$, $\norm{T_0}=1$ and $(\vv_0,x_0)\in \Pi(\ell_1)$ such that
\[
    \norm{T^*-T_0}\leq \ee,\,\,\,\, \norm{\vv-\vv_0}\leq\ee,\,\,\,\,
    \norm{x-x_0}\leq\ee
\]
and {$x_0(T_0\vv_0)=1$.}

We assert that $(x_0, \vv_0)$  is the pair we are looking for. To show this, we will reexamine the proof
of {Theorem~\ref{BPBP-l1}} to establish how $x_0$, $\vv_0$ and $T_0$ are defined.
{Indeed, from \eqref{eq:TheoP}, \eqref{eq:TeoVector}, \eqref{eq:Funcional-bis} and \eqref{eq:operator} we have respectively
\begin{align}
P&=\mathcal{P}_{(\vv, T^{**}x)}(\ee^3/480),\nonumber\\
\vv_0&=\norm{\vv \cdot \mathbbm{1}_P}^{-1}\cdot \vv \cdot \mathbbm{1}_P,\nonumber\\
x_0&=x\cdot \mathbbm{1}_{{\mathbb{N}\setminus A_{x}(\ee^2/80)}}+e^{{\arg}(x)i}\cdot\mathbbm{1}_{A_{x}(\ee^2/80)},\label{eq:x0}\\
T_0x&=\mathbbm{1}_{\mathbb{N}\setminus P}\cdot Tx+\sum_{j\in P}e_j^*(x)z_j, \text{ for } x\in\ell_1,\nonumber
\end{align}
where $\{z_j\}_{j\in P}\subset\pi_1(\vv_0)$.}

Note that $\mathcal{A}_{x}(\ee^2/80)=\{j\in N\colon  |x_j|\geq 1-\ee^2/80\}$ and that $x\in c_0$. Thus, $A_{x}(\ee^2/80)$ is finite which, by~\eqref{eq:x0}, implies that $x_0\in c_0$.

{We shall show that $T_0$ is an adjoint operator and thus that there exists $S\in \mathfrak{L}(c_0)$ such that $S^*=T_0$. It will be enough to show that $T_0^*|_{c_0}\subset c_0$.} Set  $t_{ij}=\langle e_i,T(e_j)\rangle$ for $i,j\in\mathbb{N}$. Fix $i\in\mathbb{N}$, then for $j\in\mathbb{N}$
\[
    \langle e_j,T_0^*(e_i)\rangle = \left\{
     \begin{array}{lr}
       t_{ji} & \text{if } j \notin P,\\
         (z_j)_{i} & \text{if } j \in P.
     \end{array}
   \right.
\]

Since $x\in c_0$, $T^{**}x$ belongs to $c_0$, which implies that {$P$ is finite}. Accordingly,  only finitely many terms of the form $\langle e_j,T_0^*(e_i)\rangle$ differ from the corresponding $t_{ji}$. On the other hand, since $T$ belongs to $\mathfrak{L}(c_0)$, it holds that $\lim_j|t_{ji}|=0$. Therefore, we deduce that $|\langle e_j,T_0^*(e_i)\rangle|\rightarrow 0$ when $j\rightarrow \infty$. This implies that $T_0^*e_i\in c_0$ and, since $i\in\mathbb{N}$ is arbitrarily chosen, we deduce that $T_0^*|_{c_0}\subset c_0$.

Hence we obtain the  operator
$S=T_0^{*}|_{c_0}\in\mathfrak{L}(c_0)$ and the pair $(x_0,\vv_0)\in \Pi(c_0)$ satisfying:
\[
    \vv_0(Sx_0)=S^*\vv_0(x_0)=x_0(S^*\vv_0)=x_0(T_0\vv_0)=1,
\]
and
\[
    \norm{S-T}=\norm{(S-T)^*}=\norm{S^*-T^*}=\norm{T_0-T^*}\leq \ee,
\]
which finishes the proof.
\end{proof}

{Theorem~\ref{theo:BPBPc0} implies the following two corollaries.
\begin{cor}\label{NR-c0}
The Banach space $c_0$ has the Bishop-Phelps-Bollob\'as property for numerical radius.
\end{cor}

\begin{cor}[{\cite{Card-c0}}]
The set ${\rm NRA}(c_0)$ is dense in $\mathfrak{L}(c_0)$.
\end{cor}

}

\section{Generalizations and remarks}\label{sec:BPBp-nuProperty-generalizations}
All the results that have been presented in sections~\ref{sec:ConstructiveVersions},~\ref{sec:BPBp-nuProperty-l1} and~\ref{sec:BPBp-nuProperty-c0} were stated and proved for the Banach spaces $\ell_1(\mathbb{C})$ or $c_0(\mathbb{C})$. However, a glance at their proofs suffices to convince oneself of their validity for $\ell_1(\mathbb{R})$ and $c_0(\mathbb{R})$ --shorter proofs and better estimates can be obtained in this case. More generally, given a non-empty set $\Gamma$ and $\mathbb{K}\in\{\mathbb{R},\mathbb{C}\}$, these results are, after suitable adjustments, still valid for $\ell_1(\Gamma,\mathbb{K})$ and $c_0(\Gamma,\mathbb{K})$. The spaces $\ell_1(\Gamma,\mathbb{K})$ and $c_0(\Gamma,\mathbb{K})$ are, respectively, the $\ell_1$-sum and the $c_0$-sum of $\Gamma$ copies of the field $\mathbb{K}$. Note that in particular $\ell_1(\mathbb{N},\mathbb{K})=\ell_1(\mathbb{K})$.

{The Banach space $c_0(\Gamma,\mathbb{K})$ is a predual of $\ell_1(\Gamma,\mathbb{K})$. Observe that both $c_0(\Gamma,\mathbb{K})$ and $\ell_1(\Gamma,\mathbb{K})$ have numerical index $1$. Previous considerations imply that both of them also have the BPB property for numerical radius. The $\weak^*$ topology of  $\ell_1(\Gamma,\mathbb{K})$ stands here for the topology induced on $\ell_1(\Gamma,\mathbb{K})$ by pointwise convergence on elements of $c_0(\Gamma,\mathbb{K})$.}

On the other hand, the proof of Theorem~\ref{theo:BPBPc0} {shows} that in Theorem~\ref{BPBP-l1} we {proved more than was stated}. Indeed, putting together Theorem~\ref{BPBP-l1}, the {ideas} on duality in the proof of Theorem~\ref{theo:BPBPc0} and considerations above, one easily proves the following theorem.

\begin{theo}\label{BPBP-l1-strong}
Let $T\in S_{\mathfrak{L}({\ell_1(\Gamma,\mathbb{K})})}$, $0<\ee<1$ and $(x,\vv)\in\Pi({\ell_1(\Gamma,\mathbb{K})})$ such that $\vv(Tx)\geq1-(\ee/9)^{9/2}$. Then
there exist $T_0\in S_{\mathfrak{L}({\ell_1(\Gamma,\mathbb{K})})}$ and $(x_0,\vv_0)\in\Pi({\ell_1(\Gamma,\mathbb{K})})$ such {that}
\[
\norm{T-T_0}\leq\ee,\,\, \norm{x-x_0}\leq\ee,\,\, \norm{\vv-\vv_0}\leq\ee\,\,\text{ and }\,\,\vv_0(T_0x_0)=1.
\]

Moreover, if $T$ is $\weak^*$-$\weak^*$-continuous and $\vv$ is $\weak^*$-continuous, then $T_0$ and $\vv_0$ will be $\weak^*$-$\weak^*$-continuous and $\weak^*$-continuous, respectively.
\end{theo}

Below are two consequences of Theorem~\ref{BPBP-l1-strong}.
\begin{theo}
The Banach space {$\ell_1(\Gamma,\mathbb{K})$} has the BPB property for numerical radius.
\end{theo}

\begin{theo}\label{TeoNew}
The Banach space ${c_0(\Gamma,\mathbb{K})}$ has the BPB property for numerical radius.
\end{theo}
\begin{proof}
Fix {$0<\varepsilon<1$}, $\delta\leq (\ee/9)^{9/2}$, $T\in S_{\mathfrak{L}(c_0(\Gamma,\mathbb{K}))}$ {and $(x,x^*)\in \Pi(c_0(\Gamma,\mathbb{K}))$} such that $x^*(Tx){\geq}1-\delta$. Applying Theorem~\ref{BPBP-l1-strong} to the $\weak^*$-$\weak^*$-continuous operator $T^*\in S_{\mathfrak{L}({\ell_1(\Gamma,\mathbb{K})})}$, the pair $(x^*,x)$ and $\ee$, gives a new $T_0\in S_{\mathfrak{L}({c_0(\Gamma,\mathbb{K})})}$ and a new pair $(x^*_0,x_0^{**})\in \Pi(\ell_1(\Gamma,\mathbb{K}))$ satisfying
\begin{equation}\label{eq:conditions}
\big\|T^*-T_0^*\big\|\leq\ee,\,\,\norm{x-x_0^{**}}\leq\ee,\,\,\norm{x^*-x^*_0}\leq\ee\text{ and }x_0^{**}\big(T_0^* x_0^*\big)=1.
\end{equation}
Moreover, $x_0^{**}$ is $\weak^*$-continuous, so we can identify it with some $x_0\in S_{c_0(\Gamma,\mathbb{K})}$. Therefore, conditions in~\eqref{eq:conditions} become
\begin{equation*}\label{eq:conditionsb}
\big\|T-T_0\big\|\leq\ee,\,\,\norm{x-x_0}\leq\ee,\,\,\norm{x^*-x^*_0}\leq\ee\text{ and }x_0^*\big(T_0 x_0\big)=1.
\end{equation*}
which are the requirements~{\eqref{eq:defBPBpnu}} in Definition~\ref{BPBp}. Consequently, $c_0(\Gamma,\mathbb{K})$ has the Bishop-Phelps-Bollob\'as property for numerical radius.
\end{proof}

{\bf Acknowledgements:} We would like to thank Professor R. M. Aron for suggesting {the quantitative point of view of numerical radius attaining operators} and for helpful conversations on this matter. We also would like to thank Professor M. Maestre for his suggestions on the transition from the real to the complex case. The first named author is particularly grateful {to} them because of their support {during} his stay at Kent State University.


\bibliography{BPBp-nu}

\providecommand{\bysame}{\leavevmode\hbox to3em{\hrulefill}\thinspace}
\providecommand{\MR}{\relax\ifhmode\unskip\space\fi MR }
\providecommand{\MRhref}[2]{%
  \href{http://www.ams.org/mathscinet-getitem?mr=#1}{#2}
}
\providecommand{\href}[2]{#2}
\begin{thebibliography}{AAGM08}

\bibitem[AAGM08]{AAGM}
Mar{\'{\i}}a~D. Acosta, Richard~M. Aron, Domingo Garc{\'{\i}}a, and Manuel
  Maestre, \emph{The {B}ishop-{P}helps-{B}ollob\'as theorem for operators}, J.
  Funct. Anal. \textbf{254} (2008), no.~11, 2780--2799. \MR{2414220
  (2009c:46016)}

\bibitem[AAP92]{AAP}
Mar{\'{\i}}a~D. Acosta, Francisco~J. Aguirre, and Rafael Pay{\'a}, \emph{A
  space by {W}. {G}owers and new results on norm and numerical radius attaining
  operators}, Acta Univ. Carolin. Math. Phys. \textbf{33} (1992), no.~2, 5--14.
  \MR{1287219 (95i:47009)}

\bibitem[Aco90]{AcostaThesis}
M.~Acosta, \emph{Operadores que alcanzan su radio num{\'e}rico}, Ph.D. thesis,
  Universidad de Granada, 1990.

\bibitem[Aco93]{ARenorming}
Mar{\'{\i}}a~D. Acosta, \emph{Every real {B}anach space can be renormed to
  satisfy the denseness of numerical radius attaining operators}, Israel J.
  Math. \textbf{81} (1993), no.~3, 273--280. \MR{1231192 (94h:47003)}

\bibitem[AP89]{APsecondadj}
Mar{\'{\i}}a~D. Acosta and Rafael Paya, \emph{Denseness of operators whose
  second adjoints attain their numerical radii}, Proc. Amer. Math. Soc.
  \textbf{105} (1989), no.~1, 97--101. \MR{937841 (89f:47008)}

\bibitem[AP93]{APRNP}
Mar{\'{\i}}a~D. Acosta and Rafael Pay{\'a}, \emph{Numerical radius attaining
  operators and the {R}adon-{N}ikod\'ym property}, Bull. London Math. Soc.
  \textbf{25} (1993), no.~1, 67--73. \MR{1190367 (93j:47005)}

\bibitem[BD71]{BD1}
F.~F. Bonsall and J.~Duncan, \emph{Numerical ranges of operators on normed
  spaces and of elements of normed algebras}, London Mathematical Society
  Lecture Note Series, vol.~2, Cambridge University Press, London, 1971.
  \MR{0288583 (44 \#5779)}

\bibitem[BD73]{BD2}
\bysame, \emph{Numerical ranges. {II}}, Cambridge University Press, New York,
  1973, London Mathematical Society Lecture Notes Series, No. 10. \MR{0442682
  (56 \#1063)}

\bibitem[Bol70]{Bollob}
B{\'e}la Bollob{\'a}s, \emph{An extension to the theorem of {B}ishop and
  {P}helps}, Bull. London Math. Soc. \textbf{2} (1970), 181--182. \MR{0267380
  (42 \#2282)}

\bibitem[BS84]{BergSims}
I.~D. Berg and Brailey Sims, \emph{Denseness of operators which attain their
  numerical radius}, J. Austral. Math. Soc. Ser. A \textbf{36} (1984), no.~1,
  130--133. \MR{720006 (84j:47004)}

\bibitem[Car85a]{Card-USm}
Carmen~Silvia Cardassi, \emph{Density of numerical radius attaining operators
  on some reflexive spaces}, Bull. Austral. Math. Soc. \textbf{31} (1985),
  no.~1, 1--3. \MR{772627 (86e:47005)}

\bibitem[Car85b]{Card-c0}
\bysame, \emph{Numerical radius attaining operators}, Banach spaces
  ({C}olumbia, {M}o., 1984), Lecture Notes in Math., vol. 1166, Springer,
  Berlin, 1985, pp.~11--14. \MR{827753 (87c:47004)}

\bibitem[Car85c]{Card-CK}
\bysame, \emph{Numerical radius-attaining operators on {$C(K)$}}, Proc. Amer.
  Math. Soc. \textbf{95} (1985), no.~4, 537--543. \MR{810159 (87a:47066)}

\bibitem[KMP06]{Kad-Mart-Paya}
Vladimir Kadets, Miguel Mart{\'{\i}}n, and Rafael Pay{\'a}, \emph{Recent
  progress and open questions on the numerical index of {B}anach spaces},
  RACSAM. Rev. R. Acad. Cienc. Exactas F\'\i s. Nat. Ser. A Mat. \textbf{100}
  (2006), no.~1-2, 155--182. \MR{2267407 (2007h:46011)}

\bibitem[Lin63]{Lindens}
Joram Lindenstrauss, \emph{On operators which attain their norm}, Israel J.
  Math. \textbf{1} (1963), 139--148. \MR{0160094 (28 \#3308)}

\bibitem[Pay92]{Paya}
Rafael Pay{\'a}, \emph{A counterexample on numerical radius attaining
  operators}, Israel J. Math. \textbf{79} (1992), no.~1, 83--101. \MR{1195254
  (93j:47004)}

\bibitem[Sim72]{Sims}
B.~Sims, \emph{On numerical range and its applications to banach algegras},
  Ph.D. thesis, University of Newcastle, Australia, 1972.

\end{thebibliography}
\bibliographystyle{amsalpha}

\end{document}